\documentclass[sn-mathphys]{sn-jnl}

\usepackage{amsmath,amscd}
\usepackage{fancyhdr}

\numberwithin{equation}{section}

\jyear{2021}%

\theoremstyle{thmstyleone}%
\newtheorem{theorem}{Theorem}[section]
\newtheorem{proposition}[theorem]{Proposition}%
\newtheorem{remark}[theorem]{Remark}

\begin{document}

\title[]{Sharp Stability of Solitons for the Cubic-Quintic NLS on $\mathbb{R}^{2}$}


\author[3]{\fnm{Yi} \sur{Jiang}}\email{yijiang103@sicnu.edu.cn}

\author[2]{\fnm{Chenglin} \sur{Wang}}\email{wangchenglinedu@163.com}

\author[1]{\fnm{Yibin} \sur{Xiao}}\email{xiaoyb9999@uestc.edu.cn}

\author*[1]{\fnm{Jian} \sur{Zhang}}\email{zhangjian@uestc.edu.cn}

\author[3]{\fnm{Shihui} \sur{Zhu}}\email{shihuizhumath@163.com}

\affil*[1]{\orgdiv{School of Mathematical Sciences}, \orgname{University of Electronic Science and Technology of China}, \orgaddress{\street{} \city{Chengdu}, \postcode{611731}, \state{} \country{China}}}

\affil[2]{\orgdiv{School of Science}, \orgname{Xihua University}, \orgaddress{\street{} \city{Chengdu}, \postcode{610039}, \state{} \country{China}}}

\affil[3]{\orgdiv{School of Mathematical Sciences}, \orgname{Sichuan Normal University}, \orgaddress{\street{} \city{Chengdu}, \postcode{610066}, \state{} \country{China}}}


\abstract{This paper concerns with the cubic-quintic nonlinear Schr\"{o}dinger equation on $\mathbb{R}^{2}$. A family of new variational problems related to the solitons are introduced and solved. Some key monotonicity and uniqueness results are obtained. Then the orbital stability of solitons at every frequency are proved in terms  of the Cazenave and Lions' argument. And classification of normalized ground states is first presented.  Our results settle the questions raised by Lewin and Rota Nodari as well as Carles and Sparber.}


\pacs[MSC classfication(2010)]{35Q55, 35C08, 37K40, 35J50, 35J60}

\maketitle
\newpage
\tableofcontents 

\pagestyle{fancy} 
\fancyhf{}
\fancyhead[OL]{ }
\fancyhead[OC]{}
\fancyhead[OR]{Y. Jiang, C. Wang, Y. Xiao, J. Zhang and S.  Zhu}
\fancyhead[EL]{Sharp stability of solitons}
\fancyhead[EC]{ }
\fancyhead[ER]{ }
\fancyfoot[RO, LE]{\thepage}

\section{Introduction}\label{sec1}

This paper studies the cubic-quintic nonlinear Schr\"{o}dinger equation (NLS)
\begin{equation}\label{1.1}
i\partial_{t}\varphi +\Delta  \varphi +\lvert\varphi\rvert^2\varphi-\lvert\varphi\rvert^4\varphi=0,\quad (t,x)\in\mathbb{R}\times\mathbb{R}^{2}.
\end{equation}
(\ref{1.1}) is the focusing mass-critical nonlinear Schr\"{o}dinger equation with the defocusing quintic perturbation on $\mathbb{R}^{2}$. We are interested in orbital stability of solitons at every frequency \cite{CS2021,JJLV2022,T2009}.

(\ref{1.1}) is a two-dimensional soliton model \cite{CGKM2012,PSM2022} present in many branches of physics, from beams carrying angular momentum \cite{UT2010} to remarkable Bose-Einstein condensation \cite{ARVK2001}, among many other fields \cite{AH1975,CV2017,LS1998, GBMRSTK2004}. The stability of solitons is a subject of external physical concern \cite{CGKM2012,PSM2022,T2009,PPT1979,CERS1986}.

From a mathematical point of view, the focusing cubic or quintic nonlinear Schr\"{o}dinger equation on $\mathbb{R}^{2}$ possesses  solitons, but all solitons are unstable (see \cite{BC1981,OT1991,W1983}). The defocusing cubic or quintic nonlinear Schr\"{o}dinger equation on $\mathbb{R}^{2}$ has no any solitons, but possesses scattering property (see \cite{KTV2009}). In addition, the scaling invariance of the pure power case is broken in (\ref{1.1}) (see \cite{C2003}). Therefore the dynamics of (\ref{1.1}) becomes a challenging issue (see \cite{CS2021,LR2020,TVZ2007,JL2022b,M2021,BS2017,Z2000,Z2000b}). This motivates us to study the stability of solitons for (\ref{1.1}) (see \cite{T2009}), which  directly  concerns about the conjectures raised by Lewin and Rota Nodari \cite{LR2020} as well as Carles and Sparber \cite{CS2021}.
\par In the energy space $H^{1}(\mathbb{R}^{2})$, consider the scalar field equation for $\omega\in\mathbb{R}$,
\begin{equation}\label{1.2}
-\Delta u-\lvert u\rvert^2u+\lvert u\rvert^4u+\omega u =0, \quad u\in H^{1}(\mathbb{R}^{2}).
\end{equation}
From Berestycki and Lions \cite{BL1983}, if and only if
\begin{equation}\label{1.3}
0<\omega<\frac{3}{16},
\end{equation}
(\ref{1.2}) possesses non-trivial solutions (also see \cite{CS2021}). From Gidas, Ni and Nirenberg \cite{GNN1979}, every positive solution of (\ref{1.2}) is radially symmetric. From Serrin and Tang \cite{ST2000}, the positive solution of (\ref{1.2}) is unique up to translations. Therefore one concludes that for $\omega\in(0,\frac{3}{16})$, (\ref{1.2}) possesses a unique positive solution $P_\omega(x)$ (see \cite{CS2021,K2011}), which is called ground state of (\ref{1.2}).

Let $P_\omega$ be a ground state of (\ref{1.2}) with $\omega\in(0,\frac{3}{16})$. It is easily checked that
\begin{equation}\label{1.4}
\varphi(t,x)=P_\omega(x)e^{i\omega t}
\end{equation}
is a solution of (\ref{1.1}), which is called a ground state soliton of (\ref{1.1}). We also directly call (\ref{1.4}) soliton of (\ref{1.1}), and call $\omega$ frequency of soliton. It is known that (\ref{1.1}) admits time-space translation invariance, phase invariance and Galilean invariance. Then for arbitrary $x_0\in\mathbb{R}^{2}$, $v_0\in\mathbb{R}^{2}$ and $\nu_0\in\mathbb{R}$, in terms of (\ref{1.4}) one has that
\begin{equation}\label{1.5}
 \varphi(t,x)=P_\omega(x-x_0-v_0t)e^{i(\omega t+\nu_0+\frac{1}{2}v_0x-\frac{1}{4}\lvert v_0\rvert^2t)}
\end{equation}
is also a soliton of (\ref{1.1}).
By \cite{S2014,S2021,CS2021,JL2022b,LR2020,T2009,CKS2023,Z2000,Z2000b}, orbital stability of solitons with regard to every frequency  for (\ref{1.1}) is a crucial open problem.

So far there are two ways to study stability of solitons for nonlinear  Schr\"{o}dinger equations (refer to \cite{L2009}). One is variational approach originated from Cazenave and Lions \cite{CL1982}. The other is spectrum approach originated from Weinstein \cite{W1985,W1986} and then considerably generalized by Grillakis, Shatah and Strauss \cite{GSS1987,GSS1990}. Both approaches have encountered essential difficulties to (\ref{1.1}), since (\ref{1.1}) fails in both scaling invariance and effective spectral analysis \cite{F2003,O1995,CS2021}. We need develop new methods to study stability of solitons for (\ref{1.1}).

In the following, we denote $\int_{\mathbb{R}^2}\cdot\ dx$ by $\int \cdot \ dx$. For $u\in H^1(\mathbb{R}^2)$, define the mass functional
\begin{equation}\label{1.6}
M(u)=\int\lvert u\lvert^{2}dx,
\end{equation}
and define the energy functional
\begin{equation}\label{1.7}
E(u)=\int \frac{1}{2}\lvert \nabla u\rvert^2-\frac{1}{4}\lvert u\rvert^4+\frac{1}{6}\lvert u\rvert^6dx.
\end{equation}

Our studies are also concerned with the normalized solution of (\ref{1.2}), which is defined as a non-trivial solution of (\ref{1.2}) satisfying the prescribed mass $M(u)=m$ (see \cite{JL2021}). Besides motivations in mathematical physics, normalized solutions are also of interest in the framework of ergodic Mean Field Games system \cite{CV2017}. Recent studies on normalized solutions refer to \cite{BJS2016,BMRV2021,JL2021,WW2022,S2021} and the references there. We can establish correspondences between the soliton frequency and the prescribed mass. Then we present first a classification of normalized solutions.

Inspired by Killip, Oh, Pocovnicu and Visan \cite{KOPV2017}, for $0<\alpha<\infty$ and $u\in H^1(\mathbb{R}^2)\backslash\{0\}$, we define the functional
\begin{equation}\label{1.8}
F_{\alpha}(u)=\frac{\lvert \lvert\nabla u\lvert \lvert_{L^{2}}^{\frac{2}{1+\alpha}}\lvert \lvert u\lvert \lvert_{L^{2}}^{\frac{2+\alpha}{1+\alpha}}\lvert \lvert u\lvert \lvert_{L^{6}}^{\frac{3\alpha}{1+\alpha}}}{\lvert \lvert u\lvert \lvert_{L^{4}}^{4}}.
\end{equation}
Then we introduce a family of variational problems
\begin{equation}\label{1.9}
C_{\alpha}=\mathop{\mathrm{inf}}_{\{u\in H^1(\mathbb{R}^2)\backslash\{0\}\}} F_{\alpha}(u).
\end{equation}

It is shown that for $0<\alpha<\infty$, $C_{\alpha}$ is achieved at some positive symmetric minimizer $Q_{\alpha}$. Moreover $Q_{\alpha}$ satisfies the  Euler-Lagrange equation corresponding to (\ref{1.9}), which is the same as (\ref{1.2}) with certain $\omega\in(0,\frac{3}{16})$.


In terms of Weinstein \cite{W1983}, let $q(x)$ be the cubic nonlinear ground state, i.e. the unique positive solution to
\begin{equation}\label{1.10}
-\Delta q+q-q^{3}=0,\;\;\;x\in\mathbb{R}^{2}.
\end{equation}
Then $\int q^{2}dx$ is a identified number (see \cite{CS2021}).
%

 Now consider the constrained variational problem to $m>0$,
\begin{equation}\label{1.11}
E_{min}(m)=\mathop{\mathrm{inf}}_{\{u\in H^1(\mathbb{R}^{2}),\;\;  M(u)=m\}} E(u).
\end{equation}
It is shown that $E_{min}(m)$ is achieved for $m>\int q^{2}dx$ (see \cite{CS2021}).

We impose the initial data of (\ref{1.1}) as follows.
\begin{equation}\label{1.12}
\varphi(0,x)=\varphi_{0}(x),\;\;\;x\in\mathbb{R}^{2}.
\end{equation}
According to Cazenave \cite{C2003}, for any $\varphi_{0}\in H^{1}(\mathbb{R}^{2})$, the Cauchy problem (\ref{1.1})-(\ref{1.12}) has a unique global solution $\varphi\in \mathcal{C}(\mathbb{R}; H^{1}(\mathbb{R}^{2}))$. The solution obeys the conservation of mass and energy. If in addition
\begin{equation}\label{1.13}
\varphi_{0}\in \Sigma:=\{u\in H^{1}(\mathbb{R}^{2}),\; x\mapsto \lvert x\lvert u\in L^{2}(\mathbb{R}^{2})\},
\end{equation}
then $\varphi\in \mathcal{C}(\mathbb{R};\Sigma)$.

From Carles and Sparber \cite{CS2021} (also see Murphy \cite{M2021}), if $\varphi_{0}\in \Sigma$ with
\begin{equation}\label{1.14}
\int \lvert \varphi_{0}\lvert^{2}dx\leq \int q^{2}dx,
\end{equation}
then the solution $\varphi\in\mathcal{C}(\mathbb{R};\Sigma)$ of the Cauchy problem (\ref{1.1})-(\ref{1.12}) is asymptotically linear, i.e. there exist $\varphi_{\pm}\in\Sigma$ such that
\begin{equation}\label{1.15}
\lvert\lvert e^{-it\Delta}\varphi(t,\cdot)-\varphi_{\pm}\lvert\lvert_{\Sigma} \underset{t \rightarrow \pm\infty}{\longrightarrow} 0.
\end{equation}

We say the orbital stability of solitons in terms of frequency means that the $P_{\omega}$-orbit $\{e^{i\omega t}P_{\omega},\;t\in\mathbb{R}\}$ for every $\omega\in (0,\frac{3}{16})$ is stable, that is, for all $\varepsilon>0$ there exists $\delta>0$ with the following property. If $\|\varphi_{0}-P_{\omega}\|_{H^{1}(\mathbb{R}^{2})}<\delta$ and $\varphi(t)$ is a solution of (\ref{1.1}) in some interval $[0,t_{0})$ with $\varphi(0)=\varphi_{0}\in H^{1}(\mathbb{R}^{2})$, then $\varphi(t)$ can be continued to a solution in $0\leq t<\infty$ and
\begin{equation}\label{1.16}
\sup_{0<t<\infty}\inf_{(s,y)\in\mathbb{R}\times\mathbb{R}^{2}}\|\varphi(t)-e^{is}P_{\omega}(\cdot-y)\|_{H^{1}(\mathbb{R}^{2})}<\varepsilon
\end{equation}
for every $\omega\in (0,\frac{3}{16})$. Otherwise the $P_{\omega}$-orbit is called unstable (see \cite{CL1982,GSS1987}).

\par The main results of this paper read as follow:

%

\begin{theorem}\label{t1.1}
Let $q(x)$ be the unique positive solution of (\ref{1.10}). Then the positive minimizer of variational problem (\ref{1.11}) with $m>\int q^{2}dx$ is unique up to translations.
\end{theorem}


\begin{theorem}\label{t1.2}
Let $\omega\in (0, \frac{3} {16})$ and $P_{\omega}(x)$ be the positive solution of (\ref{1.2}). Then the soliton  $e^{i\omega t}P_{\omega}$ of (\ref{1.1}) is orbitally stable for every $\omega\in (0,\frac{3}{16})$.
\end{theorem}

Theorem \ref{t1.1} gets uniqueness of the positive minimizer for $E_{min}(m)$. Thus Theorem \ref{t1.1} settles the questions raised by \cite{JL2022a} and \cite{KOPV2017}, where uniqueness of the energy minimizer is proposed. We see that uniqueness of the energy minimizer is concerned with the monotonicity of $M(P_{\omega})$ for $\omega \in (0,\frac{3}{16})$ and stability of solitons. In fact, uniqueness of positive minimizer of the variational problem is remarkable \cite{LR2020,Z2000,KOPV2017}.


Theorem \ref{1.2} gives orbital stability of solitons of (\ref{1.1}) for every frequency $\omega\in (0,\frac{3}{16})$ corresponding to the mass $m>\int q^{2}dx$. On the other hand, for any initial data $\varphi_{0}(x)$ such that $\varphi_{0}(x)\in\Sigma$ and $m=\int \lvert \varphi_{0}\lvert^{2}dx\leq \int q^{2}dx$, one has that the solutions of the Cauchy problem (\ref{1.1})-(\ref{1.12}) are scattering (see \cite{CS2021}). Therefore we say Theorem \ref{t1.2} gives sharp stability of solitons of (\ref{1.1}) in terms of every frequency.

We see that Theorem \ref{t1.2} settles the questions raised by \cite{BS2017,CS2021,JL2022a,JL2022b,LR2020,S2021,S2014}, where the orbital stability of solitons is only with regard to the set of ground states, that is actually a weak stability on solitons. Then Theorem \ref{t1.2} also answers the open problem on stability of solitons proposed by Tao \cite{T2009}.

Uniqueness of the energy minimizers plays a key role in the proof of Theorem \ref{t1.2}. Lewin and Rota Nodari have pointed out this fact in \cite{LR2020}, also see \cite{CS2021}. Applying this uniqueness we also give a complete classification about normalized solutions of (\ref{1.2}), which depends on establishing the correspondence between the frequency and the mass. Thus we settle the questions raised by \cite{BMRV2021,BS2017,JL2022b,S2021}. The method developed in this paper is universal and can be used to solve many soliton stability problems that do not have scale invariance, such as nonlinear Schr\"{o}dinger equations with harmonic potential or Hardy potential, Hartree equations, Davey-Stewartson system, Inhomogeneous nonlinear Schr\"{o}dinger equation etc.

In terms of \cite{CS2021,LR2020,M2021} and Theorem \ref{t1.2} in this paper, the global dynamics of (\ref{1.1}) are more comprehensively described. By Theorem \ref{t1.2} in this paper, according to \cite{MM2006,MMT2006} multi-solitons of (\ref{1.1})
can be constructed (also refer to \cite{CL2011,BZ2022}). But the soliton resolution conjecture for (\ref{1.1}) is still open (see \cite{T2009}).

\par This paper is organized as follows. In section 2, we induce some propositions of ground states based on \cite{CS2021}. In section 3, we introduce and solve a family of variational problems inspired by \cite{KOPV2017}. In section 4, we prove some key results about monotonicity and uniqueness. In section 5, we prove sharp stability of the solitons. In section 6, we present classification of normalized solution of (\ref{1.2}) for the first time.

\section{Ground states}\label{sec2}


Bsaed on Carles and Sparber \cite{CS2021}, Lewin and Rota Nodari \cite{LR2020} as well as Coddington and Levinson \cite{CL1955}, the following propositions are true.
\begin{proposition}\label{p2.2}
(\cite{CS2021}) (\ref{1.2}) possesses a positive solution $P_\omega$ if and only if $\omega\in(0,\frac{3}{16})$. In addition, $P_\omega$  holds the following properties.\\
(I) $P_\omega$ is radially symmetric and unique up to translations.\\
(II) $P_\omega(x)$ is a real-analytic function of $x$ and that for some $c=c(\omega)>0$ as $\lvert x\rvert\rightarrow \infty$,
\begin{equation*}
\lvert x\rvert \mathrm{exp}\{\sqrt{\omega}\lvert x\rvert\}P_\omega(x)\rightarrow c,\quad
\mathrm{exp}\{\sqrt{\omega}\lvert x\rvert\}x \nabla P_\omega(x)\rightarrow -\sqrt{\omega}c.
\end{equation*}
(III) Pohozaev identity
\begin{equation*}
\int\lvert \nabla P_\omega \rvert^2 + \frac{2}{3}P_\omega^6-\frac{1}{2}P_\omega^4dx=0.
\end{equation*}
\end{proposition}
\begin{proposition}\label{p2.3}
(\cite{CS2021}) Let $\omega\in(0,\frac{3}{16})$ and $P_\omega$ be the ground state of (\ref{1.2}). One has both the map $\omega\rightarrow M(P_\omega)$ and the map $\omega\rightarrow E(P_\omega)$ are $\mathcal{C}^1$, indeed, $M(P_\omega)$ and $E(P_\omega)$ are real analytic. Moreover the followings are true: \\
\begin{equation*}
\frac{d}{d\omega}E(P_\omega)=-\frac{\omega}{2}\frac{d}{d\omega}M(P_\omega);
\end{equation*}
\begin{equation*}
M(P_\omega)\rightarrow\infty \ and \ E(P_\omega)\rightarrow-\infty, \ as \ \omega\rightarrow\frac{3}{16};
\end{equation*}
\begin{equation*}
M(P_\omega)\rightarrow \int q^{2}dx \ as \ \omega\rightarrow0,\\
\end{equation*}
where $q(x)$ is the unique positive solution of (\ref{1.10}).
\end{proposition}

\begin{remark}
For the one-dimensional cubic-quintic nonlinear Schr\"{o}dinger equation, the corresponding ground state satisfies the following one-dimensional nonlinear elliptic equation
\begin{equation*}
-\Delta \phi+\omega\phi -\phi^{3}+\phi^{5}=0,\;\;\;\;\phi\in H^{1}(\mathbb{R})\backslash\{0\},
\end{equation*}
where $\omega>0$ is frequency of the soliton. From \cite{CERS1986,PPT1979}, one has that
\begin{equation*}
\phi(x)=2\sqrt{\frac{\omega}{1+\sqrt{1-\frac{16}{3}\omega}\cdot cosh(2\sqrt{\omega}x)}}.
\end{equation*}
It is obvious that $\omega\in (0,\frac{3}{16})$.
\end{remark}

\section{Variational approaches}\label{sec3}

~~~~~For $0<\alpha<\infty$ and $u\in H^{1}(\mathbb{R}^{2})\backslash\{0\}$, define the functional
\begin{equation}\label{3.0a}
F_{\alpha}(u)=\frac{\lvert \lvert\nabla u\lvert \lvert_{L^{2}}^{\frac{2}{1+\alpha}}\lvert \lvert u\lvert \lvert_{L^{2}}^{\frac{2+\alpha}{1+\alpha}}\lvert \lvert u\lvert \lvert_{L^{6}}^{\frac{3\alpha}{1+\alpha}}}{\lvert \lvert u\lvert \lvert_{L^{4}}^{4}}.
\end{equation}
Then for $0<\alpha<\infty$, we introduce a family of variational problems
\begin{equation}\label{3.0b}
C_{\alpha}=\inf_{\{u\in H^{1}(\mathbb{R}^{2})\backslash\{0\}\}}F_{\alpha}(u).
\end{equation}

\begin{theorem}\label{t3.1}
Let $0<\alpha<\infty$. Then the variational problem (\ref{3.0b}) is achieved  at a non-negative radially symmetric function $v(x)\in H^{1}(\mathbb{R}^{2})$ such that
\begin{equation*}
C_{\alpha}=\min_{\{u\in H^1(\mathbb{R}^2)\backslash\{0\}\}} F_{\alpha}(u).
\end{equation*}
\end{theorem}
\begin{proof}
It is obvious that $0\leq C_{\alpha}<\infty$ for every $0<\alpha<\infty$. Let $\{u_{n}\}_{n\in N}\subset H^{1}(\mathbb{R}^{2})$ be a minimizing sequence of $C_{\alpha}$, that is
\begin{equation}\label{3.1aa}
C_{\alpha}=\lim_{n\rightarrow\infty}F_{\alpha}(u_{n}).
\end{equation}
By the P\'{o}lya-Szeg\"{o} rearrangement inequality (see \cite{LL2000}), we can assume that all functions $u_{n}$ are non-negative and radially symmetric decreasing. Note that if we set
\begin{equation}\label{3.1b}
u^{\lambda,\;\mu}(x)=\mu u(\lambda x)\;\;\;\;\lambda>0,\; \mu>0,
\end{equation}
then one has that
\begin{equation}\label{3.1al}
\lvert \lvert \nabla u^{\lambda,\mu}\lvert \lvert_{L^{2}}^{2}=\mu^{2}\lvert \lvert\nabla u\lvert \lvert_{L^{2}}^{2},\;\;\;\;\;\lvert \lvert u^{\lambda,\mu}\lvert \lvert_{L^{2}}^{2}=\lambda^{-2}\mu^{2}\lvert \lvert u\lvert \lvert_{L^{2}}^{2},
\end{equation}
\begin{equation}\label{3.1am}
\lvert \lvert u^{\lambda,\mu}\|_{L^{4}}^{4}=\lambda^{-2}\mu^{4}\|u\|_{L^{4}}^{4},\;\;\;\;\;\|u^{\lambda,\mu}\|_{L^{6}}^{6}=\lambda^{-2}\mu^{6}\|u\|_{L^{6}}^{6}.
\end{equation}
It follows that
\begin{equation}\label{3.1an}
\begin{aligned}
F_{\alpha}(u^{\lambda,\mu})=&\frac{\big(\mu^{2}\|\nabla u\|_{L^{2}}^{2}\big)^{\frac{1}{1+\alpha}}\big(\lambda^{-2}\mu^{2}\|u\|_{L^{2}}^{2}\big)^{\frac{2+\alpha}{2(1+\alpha)}}\big(\lambda^{-2}\mu^{6}\|u\|_{L^{6}}^{6}\big)^{\frac{\alpha}{2(1+\alpha)}}}{\lambda^{-2}\mu^{4}\|u\|_{L^{4}}^{4}}\\
=&\frac{\lvert \lvert\nabla u\lvert \lvert_{L^{2}}^{\frac{2}{1+\alpha}}\lvert \lvert u\lvert \lvert_{L^{2}}^{\frac{2+\alpha}{1+\alpha}}\lvert \lvert u\lvert \lvert_{L^{6}}^{\frac{3\alpha}{1+\alpha}}}{\lvert \lvert u\lvert \lvert_{L^{4}}^{4}}\\
=&F_{\alpha}(u)
\end{aligned}
\end{equation}
Thus we may assume that for all $n$,
\begin{equation}\label{3.1c}
\lvert \lvert u_{n}\lvert \lvert_{L^{2}}=1\;\;\;\;\text{and}\;\;\;\;\lvert \lvert\nabla u_{n}\lvert \lvert_{L^{2}}=1.
\end{equation}
By the Banach-Alaoglu theorem, up to a subsequence as $n\rightarrow\infty$, we can assume that
\begin{equation}\label{3.1d}
u_{n}\rightharpoonup v\;\;\;\;\text{weakly in}\;\;\; H^{1}(\mathbb{R}^{2}).
\end{equation}
Using the well-known compactness of the embedding $H_{rad}^{1}(\mathbb{R}^{2})\hookrightarrow L^{p}(\mathbb{R}^{2})$ with $p>2$, see \cite{S1977}, Sobolev embedding further guarantees that
\begin{equation}\label{3.1e}
u_{n}\rightharpoonup v\;\;\;\;\text{strongly in}\;\;\; L^{4}(\mathbb{R}^{2})\;\;\;\text{and}\;\;\;u_{n}\rightharpoonup v\;\;\;\;\text{strongly in}\;\;\; L^{6}(\mathbb{R}^{2}).
\end{equation}
\par Now we verify that $v\neq 0$. By the normalizations (\ref{3.1c}) and H\"{o}lder's inequality, we have
\begin{equation}\label{3.1f}
C_{\alpha}=\lim_{n\rightarrow\infty}F_{\alpha}(u_{n})=\lim_{n\rightarrow\infty}\frac{\lvert \lvert u_{n}\lvert \lvert_{L^{2}}^{\frac{\alpha}{1+\alpha}}\lvert \lvert u_{n}\lvert \lvert_{L^{6}}^{\frac{3\alpha}{1+\alpha}}}{\lvert \lvert u_{n}\lvert \lvert_{L^{4}}^{4}}\geq\lim_{n\rightarrow\infty}\frac{\lvert \lvert u_{n}\lvert \lvert_{L^{4}}^{\frac{4\alpha}{1+\alpha}}}{\lvert \lvert u_{n}\lvert \lvert_{L^{4}}^{4}}=\lvert \lvert v\lvert \lvert_{L^{4}}^{-\frac{4}{1+\alpha}},
\end{equation}
which shows that indeed  $v\neq 0$.
\par By the weakly lower semi-continuous of norm on any Banach space, one gets from (\ref{3.1d}) that
\begin{equation}\label{3.1g}
\lvert \lvert v\lvert \lvert_{L^{2}}\leq \liminf_{n\rightarrow\infty}\lvert \lvert u_{n}\lvert \lvert_{L^{2}}=1 \;\;\;\;\text{and}\;\;\;\;\lvert \lvert\nabla v\lvert \lvert_{L^{2}}\leq \liminf_{n\rightarrow\infty}\lvert \lvert \nabla u_{n}\lvert \lvert_{L^{2}}=1.
\end{equation}
Thus
\begin{equation}\label{3.1h}
C_{\alpha}=\lim_{n\rightarrow\infty}F_{\alpha}(u_{n})\geq F_{\alpha}(v)\geq \inf_{\{u\in H^{1}(\mathbb{R}^{2})\backslash\{0\}\}}F_{\alpha}(v)=C_{\alpha}.
\end{equation}
This equality  holds throughout this line and $v$ is an optimizer for $(\ref{3.0b})$. Since the minimizing sequence $\{u_{n}\}$ are nonnegative radially symmetric decreasing, the limit $v$ is also nonnegative radially symmetric decreasing.
\end{proof}

\begin{theorem}\label{t3.2a}
 Let $0<\alpha<\infty$. Then the variational problem (\ref{3.0b}) possesses a positive radially symmetric minimizer $Q_{\alpha}(x)$ such that
\begin{equation*}
-\Delta Q_{\alpha}+Q_{\alpha}^{5}-Q_{\alpha}^{3}+\omega Q_{\alpha}=0\;\;\;\;\text{with}\;\;\;\;\omega=\frac{2+\alpha}{2}\frac{\lvert \lvert\nabla Q_{\alpha}\lvert \lvert_{L^{2}}^{2}}{\lvert \lvert Q_{\alpha}\lvert \lvert_{L^{2}}^{2}}.
\end{equation*}
Moreover,
\begin{equation*}
\beta(Q_{\alpha})=\frac{\int \lvert Q_{\alpha} \lvert^{6}dx }{\int \lvert \nabla Q_{\alpha} \lvert^{2}dx}=\frac{3}{2}\alpha,\;\;\;\;\;C_{\alpha}=\frac{(\frac{3}{2}\alpha)^{\frac{\alpha}{2(1+\alpha)}}}{2(1+\alpha)}\lvert \lvert Q_{\alpha}\lvert \lvert_{L^{2}}^{\frac{2+\alpha}{1+\alpha}}\lvert \lvert \nabla Q_{\alpha}\lvert \lvert_{L^{2}}^{-\frac{\alpha}{1+\alpha}}.
\end{equation*}
\end{theorem}
\begin{proof}
According to Theorem \ref{t3.1}, the variational problem (\ref{3.0b}) possesses a nonnegative radially symmetric minimizer $v(x)\in H^{1}(\mathbb{R}^{2})$. Thus $v(x)$ satisfies the corresponding Euler-Lagrange equation:
\begin{equation}\label{3.2aa}
\frac{d}{d\varepsilon}\big\lvert_{\varepsilon=0}F_{\alpha}(v+\varepsilon \varphi)=0\;\;\;\text{for all}\;\;\;\varphi\in \mathcal{C}_{0}^{\infty}(\mathbb{R}^{2}).
\end{equation}
Direct computation shows that $v$ is a distributional solution to the following equations:
\begin{equation}\label{3.2ab}
-\Delta v+\frac{1+\alpha}{2}\frac{\lvert \lvert\nabla v\lvert \lvert_{L^{2}}^{2}}{\lvert \lvert v\lvert \lvert_{L^{2}}^{2}}v+\frac{3\alpha}{2}\frac{\lvert \lvert\nabla v\lvert \lvert_{L^{2}}^{2}}{\lvert \lvert v\lvert \lvert_{L^{6}}^{6}}v^{5}-4(1+\alpha)\frac{\lvert \lvert\nabla v\lvert \lvert_{L^{2}}^{2}}{\lvert \lvert v\lvert \lvert_{L^{4}}^{4}}v^{3}=0.
\end{equation}
Let $\lambda>0$, $\rho>0$ defined by
\begin{equation}\label{3.2ac}
\lambda^{2}=\frac{8(1+\alpha)}{3\alpha}\frac{\lvert \lvert v\lvert \lvert_{L^{6}}^{6}}{\lvert \lvert v\lvert \lvert_{L^{4}}^{4}}\;\;\;\;\text{and}\;\;\;\;\rho^{2}=\frac{32(1+\alpha)^{2}}{3\alpha}\frac{\lvert \lvert\nabla v\lvert \lvert_{L^{2}}^{2}\lvert \lvert v\lvert \lvert_{L^{6}}^{6}}{(\lvert \lvert v\lvert \lvert_{L^{4}}^{4})^{2}}.
\end{equation}
Denote
\begin{equation}\label{3.2ad}
Q_{\alpha}(x)=\frac{1}{\lambda}v(\frac{x}{\rho}).
\end{equation}
Then $Q_{\alpha}\geq 0$ is radially symmetric. By (\ref{3.2ab}), one has that $Q_{\alpha}$ satisfies
\begin{equation}\label{3.2ae}
-\Delta Q_{\alpha}+Q_{\alpha}^{5}-Q_{\alpha}^{3}+\omega Q_{\alpha}=0\;\;\;\;\text{with}\;\;\;\;\omega=\frac{2+\alpha}{2}\frac{\lvert \lvert\nabla Q_{\alpha}\lvert \lvert_{L^{2}}^{2}}{\lvert \lvert Q_{\alpha}\lvert \lvert_{L^{2}}^{2}}.
\end{equation}
Applying the strong maximum principle, one gets that $Q_{\alpha}>0$. By (\ref{3.2ae}), one has the Pohozaev identity 
\begin{equation}\label{3.2af}
\int \lvert \nabla Q_{\alpha}\lvert^{2}+\frac{2}{3}\lvert Q_{\alpha}\lvert^{6}-\frac{1}{2}\lvert Q_{\alpha}\lvert^{4}dx=0.
\end{equation}
From (\ref{3.2ac}) and (\ref{3.2ad}), a direct computation shows that
\begin{equation}\label{3.2ag}
\beta(Q_{\alpha})=\frac{\lvert \lvert Q_{\alpha}\lvert \lvert_{L^{6}}^{6}}{\lvert \lvert\nabla Q_{\alpha}\lvert \lvert_{L^{2}}^{2}}=\frac{\rho^{2}}{\lambda^{4}}\beta(v)=\frac{3}{2}\alpha.
\end{equation}
By (\ref{3.2af}) and (\ref{3.2ag}), it follows that
\begin{equation}\label{3.2ah}
\lvert \lvert Q_{\alpha}\lvert \lvert_{L^{6}}^{6}=\frac{3}{2}\alpha\lvert \lvert\nabla Q_{\alpha}\lvert \lvert_{L^{2}}^{2}\;\;\;\;\text{and}\;\;\;\;\lvert \lvert Q_{\alpha}\lvert \lvert_{L^{4}}^{4}=2(1+\alpha)\lvert \lvert\nabla Q_{\alpha}\lvert \lvert_{L^{2}}^{2}.
\end{equation}
By (\ref{3.0a}) and (\ref{3.2ah}), one has that
\begin{equation}\label{3.2ai}
F_{\alpha}(Q_{\alpha})=\frac{(\frac{3}{2}\alpha)^{\frac{\alpha}{2(1+\alpha)}}}{2(1+\alpha)}\lvert \lvert Q_{\alpha}\lvert \lvert_{L^{2}}^{\frac{2+\alpha}{1+\alpha}}\lvert \lvert\nabla Q_{\alpha}\lvert \lvert_{L^{2}}^{-\frac{\alpha}{1+\alpha}}.
\end{equation}
On the other hand, by (\ref{3.2ac}) and (\ref{3.2ad}), one has that
\begin{equation}\label{3.2aj}
C_{\alpha}=F_{\alpha}(v)=\frac{(\frac{3}{2}\alpha)^{\frac{\alpha}{2(1+\alpha)}}}{2(1+\alpha)}\lvert \lvert Q_{\alpha}\lvert \lvert_{L^{2}}^{\frac{2+\alpha}{1+\alpha}}\lvert \lvert\nabla Q_{\alpha}\lvert \lvert_{L^{2}}^{-\frac{\alpha}{1+\alpha}}.
\end{equation}
From (\ref{3.2ai}) and (\ref{3.2aj}),
\begin{equation}\label{3.2ak}
C_{\alpha}=F_{\alpha}(Q_{\alpha}).
\end{equation}
Therefore, $Q_{\alpha}$ is a positive minimizer of the variational problem (\ref{3.0b}). In addition,
\begin{equation}\label{3.2al}
C_{\alpha}=\frac{(\frac{3}{2}\alpha)^{\frac{\alpha}{2(1+\alpha)}}}{2(1+\alpha)}\lvert \lvert Q_{\alpha}\lvert \lvert_{L^{2}}^{\frac{2+\alpha}{1+\alpha}}\lvert \lvert\nabla Q_{\alpha}\lvert \lvert_{L^{2}}^{-\frac{\alpha}{1+\alpha}}.
\end{equation}
\end{proof}

\section{Monotonicity and uniqueness}\label{sec4}

\begin{theorem}\label{t4.1}
Let $0<\alpha<\infty$ and $Q_{\alpha}$ as in Theorem \ref{t3.2a}. Then one has that
\begin{equation*}
\alpha\longmapsto \lvert \lvert\nabla Q_{\alpha}\lvert \lvert_{L^{2}}^{2}\;\;\;\;\text{is strictly increasing on}\;\;\; (0,\infty),
\end{equation*}
\begin{equation*}
\alpha\longmapsto \lvert \lvert Q_{\alpha}\lvert \lvert_{L^{2}}^{2}\;\;\;\;\text{is strictly increasing on}\;\;\; (0,\infty).
\end{equation*}
\end{theorem}
\begin{proof}
As $\nu\neq \alpha$, Theorem \ref{t3.2a} guarantees that  $Q_{\nu}$ is not a minimizer of variational problem (\ref{3.0b}). Thus it follows that
\begin{equation}\label{4.1a}
F_{\alpha}(Q_{\nu})>F_{\alpha}(Q_{\alpha})=C_{\alpha}.
\end{equation}
Since $Q_{\nu}$ obeys the relations
\begin{equation}\label{4.1b}
\lvert \lvert Q_{\nu}\lvert \lvert_{L^{6}}^{6}=\frac{3}{2}\nu\lvert \lvert\nabla Q_{\nu}\lvert \lvert_{L^{2}}^{2}\;\;\;\;\text{and}\;\;\;\;\lvert \lvert Q_{\nu}\lvert \lvert_{L^{4}}^{4}=2(1+\nu)\lvert \lvert\nabla Q_{\nu}\lvert \lvert_{L^{2}}^{2},
\end{equation}
then one has that
\begin{equation}\label{4.1c}
F_{\alpha}(Q_{\nu})=\frac{(\frac{3}{2}\nu)^{\frac{\nu}{2(1+\nu)}}}{2(1+\nu)}\lvert \lvert Q_{\nu}\lvert \lvert_{L^{2}}^{\frac{2+\nu}{1+\nu}}\lvert \lvert\nabla Q_{\nu}\lvert \lvert_{L^{2}}^{-\frac{\nu}{1+\nu}}.
\end{equation}
By (\ref{4.1a}), (\ref{4.1c}) and Theorem \ref{t3.2a}, it follows that
\begin{equation}\label{4.1d}
\frac{\lvert \lvert Q_{\nu}\lvert \lvert_{L^{2}}^{\frac{2+\alpha}{1+\alpha}}\lvert \lvert\nabla Q_{\nu}\lvert \lvert_{L^{2}}^{-\frac{\alpha}{1+\alpha}}}{\lvert \lvert Q_{\alpha}\lvert \lvert_{L^{2}}^{\frac{2+\alpha}{1+\alpha}}\lvert \lvert\nabla Q_{\alpha}\lvert \lvert_{L^{2}}^{-\frac{\alpha}{1+\alpha}}}>\frac{(1+\nu)\alpha^{\frac{\alpha}{2(1+\alpha)}}}{(1+\alpha)\nu^{\frac{\alpha}{2(1+\alpha)}}}.
\end{equation}
By the similar manner, we also get from $F_{\nu}(Q_{\alpha})>F_{\nu}(Q_{\nu})=C_{\nu}$ that
\begin{equation}\label{4.1e}
\frac{\lvert \lvert Q_{\alpha}\lvert \lvert_{L^{2}}^{\frac{2+\nu}{1+\nu}}\lvert \lvert\nabla Q_{\alpha}\lvert \lvert_{L^{2}}^{-\frac{\nu}{1+\nu}}}{\lvert \lvert Q_{\nu}\lvert \lvert_{L^{2}}^{\frac{2+\nu}{1+\nu}}\lvert \lvert\nabla Q_{\nu}\lvert \lvert_{L^{2}}^{-\frac{\nu}{1+\nu}}}>\frac{(1+\alpha)\nu^{\frac{\nu}{2(1+\nu)}}}{(1+\nu)\alpha^{\frac{\nu}{2(1+\nu)}}}.
\end{equation}
Combining (\ref{4.1d}) and (\ref{4.1e}) gives that
\begin{equation}\label{4.1f}
\frac{(1+\alpha)^{\frac{1+\alpha}{2+\alpha}}\nu^{\frac{\alpha}{2(2+\alpha)}}}{(1+\nu)^{\frac{1+\alpha}{2+\alpha}}\alpha^{\frac{\alpha}{2(2+\alpha)}}}\Big(\frac{\lvert \lvert\nabla Q_{\alpha}\lvert \lvert_{L^{2}}}{\lvert \lvert\nabla Q_{\nu}\lvert \lvert_{L^{2}}}\Big)^{\frac{\alpha}{2+\alpha}}>\frac{\lvert \lvert Q_{\alpha}\lvert \lvert_{L^{2}}}{\lvert \lvert Q_{\nu}\lvert \lvert_{L^{2}}}>\frac{(1+\alpha)^{\frac{1+\nu}{2+\nu}}\nu^{\frac{\nu}{2(2+\nu)}}}{(1+\nu)^{\frac{1+\nu}{2+\nu}}\alpha^{\frac{\nu}{2(2+\nu)}}}\Big(\frac{\|\nabla Q_{\alpha}\lvert \lvert_{L^{2}}}{\lvert \lvert\nabla Q_{\nu}\lvert \lvert_{L^{2}}}\Big)^{\frac{\nu}{2+\nu}},
\end{equation}
from which we will derive the remaining assertions of this Theorem.
\par Skipping over the middle term in (\ref{4.1f}) and rearranging gives
\begin{equation}\label{4.1g}
\big(\frac{\lvert \lvert\nabla Q_{\alpha}\lvert \lvert_{L^{2}}}{\lvert \lvert\nabla Q_{\nu}\lvert \lvert_{L^{2}}}\big)^{\frac{2(\alpha-\nu)}{(2+\alpha)(2+\nu)}}>\big(\frac{\alpha(1+\nu)}{\nu(1+\alpha)}\big)^{\frac{\alpha-\nu}{(2+\alpha)(2+\nu)}}
\end{equation}
which then implies that for all $0<\nu<\alpha$,
\begin{equation}\label{4.1h}
\frac{\lvert \lvert\nabla Q_{\alpha}\lvert \lvert_{L^{2}}^{2}}{\lvert \lvert\nabla Q_{\nu}\lvert \lvert_{L^{2}}^{2}}>\frac{\alpha(1+\nu)}{\nu(1+\alpha)}.
\end{equation}
Therefore, it follows that
\begin{equation}\label{4.1i}
\alpha\longmapsto \lvert \lvert\nabla Q_{\alpha}\lvert \lvert_{L^{2}}^{2}\;\;\;\;\text{is strictly increasing on}\;\;(0,\infty).
\end{equation}
\par From (\ref{4.1h}) and the first inequality in (\ref{4.1f}) we deduce that
\begin{equation}\label{4.1j}
\frac{\lvert \lvert Q_{\alpha}\lvert \lvert_{L^{2}}^{2}}{\lvert \lvert Q_{\nu}\lvert \lvert_{L^{2}}^{2}}>\frac{1+\alpha}{1+\nu}\;\;\;\text{for all}\;\;\; 0<\nu<\alpha.
\end{equation}
Therefore, it follows that
\begin{equation}\label{4.1k}
\alpha\longmapsto \lvert \lvert Q_{\alpha}\lvert \lvert_{L^{2}}^{2}\;\;\;\;\text{is strictly increasing on}\;\;(0,\infty).
\end{equation}
\end{proof}

\begin{theorem}\label{t4.2}
Let $0<\alpha<\infty$. Then $C_{\alpha}$ identifies a unique $Q_{\alpha}(x)>0$ satisfying
\begin{equation*}
-\Delta Q_{\alpha}+Q_{\alpha}^{5}-Q_{\alpha}^{3}+\omega Q_{\alpha}=0
\end{equation*}
with
\begin{equation*}
\omega=\frac{2+\alpha}{2}\frac{\lvert \lvert\nabla Q_{\alpha}\lvert \lvert_{L^{2}}^{2}}{\lvert \lvert Q_{\alpha}\lvert \lvert_{L^{2}}^{2}}.
\end{equation*}
In addition, $\omega$ is completely identified by $\alpha$ and
\begin{equation*}
Q_{\alpha}=P_{\omega}\;\;\;\text{up to translations}.
\end{equation*}
\end{theorem}
\begin{proof}
According to Theorem \ref{t3.2a}, $C_{\alpha}$ admits a positive minimizer $Q_{\alpha}(x)$, which is not uniquely identified by $\alpha$. In addition, $Q_{\alpha}$ satisfies
\begin{equation}\label{4.2a}
-\Delta Q_{\alpha}+Q_{\alpha}^{5}-Q_{\alpha}^{3}+\omega Q_{\alpha}=0 
\end{equation}
with
\begin{equation}\label{4.2b}
\omega=\frac{2+\alpha}{2}\frac{\lvert \lvert\nabla Q_{\alpha}\lvert \lvert_{L^{2}}^{2}}{\lvert \lvert Q_{\alpha}\lvert \lvert_{L^{2}}^{2}}.
\end{equation}
In terms of Theorem \ref{t4.1}, one gets that
\begin{equation}\label{4.2d}
\lvert \lvert \nabla Q_{\alpha}\lvert \lvert_{L^{2}}:=f(\alpha)\;\;\;\text{is completely identified by}\;\alpha,
\end{equation}
\begin{equation}\label{4.2da}
\lvert \lvert Q_{\alpha}\lvert \lvert_{L^{2}}:=g(\alpha)>0\;\;\; \text{is completely identified by}\;\alpha.
\end{equation}
Thus
\begin{equation}\label{4.2e}
\omega=\frac{2+\alpha}{2}\frac{\lvert \lvert\nabla Q_{\alpha}\lvert \lvert_{L^{2}}^{2}}{\lvert \lvert Q_{\alpha}\lvert \lvert_{L^{2}}^{2}}:=h(\alpha)\;\;\; \text{is completely identified by}\;\alpha.
\end{equation}
From Proposition \ref{p2.2}, it follows that
\begin{equation}\label{4.2ea}
P_{\omega}(x)\;\;\; \text{is completely identified by}\;\alpha.
\end{equation}
Since
\begin{equation}\label{4.2eb}
Q_{\alpha}(x)=P_{\omega}(x)\;\;\;\;\text{up to translations},
\end{equation}
it follows that
\begin{equation}\label{4.2f}
Q_{\alpha}\;\;\; \text{is completely identified by}\;\alpha.
\end{equation}
Therefore for $0<\alpha<\infty$, $C_{\alpha}$ identifies a unique $Q_{\alpha}(x)>0$ satisfying
\begin{equation*}
-\Delta Q_{\alpha}+Q_{\alpha}^{5}-Q_{\alpha}^{3}+\omega Q_{\alpha}=0
\end{equation*}
with
\begin{equation}\label{4.2aj}
\omega=\frac{2+\alpha}{2}\frac{\lvert \lvert\nabla Q_{\alpha}\lvert \lvert_{L^{2}}^{2}}{\lvert \lvert Q_{\alpha}\lvert \lvert_{L^{2}}^{2}}\;\;\;\;\text{and}\;\;\;\;Q_{\alpha}=P_{\omega}\;\;\;\text{up to translations}.
\end{equation}
\end{proof}

\begin{theorem}\label{t4.3}
Let $m>\int q^{2}dx$, where $q(x)$ is the unique positive solution of (\ref{1.10}). Then $m$ identifies a unique $\omega\in (0,\frac{3}{16})$ such that $m=M(P_{\omega})$.
\end{theorem}
\begin{proof}
According to Theorem \ref{t4.1},
\begin{equation}\label{4.3a}
\alpha\longmapsto M(Q_{\alpha})=\lvert \lvert Q_{\alpha}\lvert \lvert_{L^{2}}^{2}\;\;\;\;\text{is strictly increasing on}\;\;\;(0,\; \infty).
\end{equation}
Then for $m>\int q^{2}dx$,
\begin{equation}\label{4.3b}
m\;\;\;\text{identifies a unique}\;\;\alpha\;\;\;\;\text{with}\;\;\;m=M(P_{\omega}).
\end{equation}
In terms of Theorem \ref{t4.2}, $\alpha$ identifies a unique $\omega$ satisfying
\begin{equation}\label{4.3c}
\omega=\frac{2+\alpha}{2}\frac{\lvert \lvert\nabla Q_{\alpha}\lvert \lvert_{L^{2}}^{2}}{\lvert \lvert Q_{\alpha}\lvert \lvert_{L^{2}}^{2}}\;\;\;\;\text{with}\;\;\;\;Q_{\alpha}=P_{\omega}\;\;\;\text{up to translations}.
\end{equation}
Combing (\ref{4.3b}) and (\ref{4.3c}), it follows that
\begin{equation}\label{4.3d}
m\;\;\;\text{identifies a unique}\;\;\omega\in(0,\frac{3}{16})\;\;\;\;\text{satisfying}\;\;\;m=M(P_{\omega}).
\end{equation}
\end{proof}

Now we consider the constrained variational problem for $m>0$,
\begin{equation}\label{4.0a}
E_{min}(m)=\inf_{\{u \in H^1(\mathbb{R}^{2}),  M(u)=m\}} E(u),
\end{equation}
where $E(u)$ and $M(u)$ as in (\ref{1.7}) and (\ref{1.6}) respectively.

\begin{theorem}\label{t4.4}
Let $m>\int q^{2}dx$, where $q(x)$ is the unique positive solution of (\ref{1.10}). Then the constrained variational problem $E_{min}(m)$ at most possesses a positive minimizer up to translations.
\end{theorem}
\begin{proof}
Suppose that $\xi_{1}$ and $\xi_{2}$ are two positive minimizers of the constrained variational problem $E_{min}(m)$ for $m>\int q^{2}dx$. Then $\xi_{1}$ satisfies the Euler-Lagrange equation with the corresponding Lagrange multiplier $\omega_{1}$
\begin{equation}\label{4.4a}
-\Delta \xi_{1}+\xi_{1}^{5}-\xi_{1}^{3}+\omega_{1}\xi_{1}=0.
\end{equation}
And $\xi_{2}$ satisfies the Euler-Lagrange equation with the corresponding Lagrange multiplier $\omega_{2}$
\begin{equation}\label{4.4a1}
-\Delta \xi_{2}+\xi_{2}^{5}-\xi_{2}^{3}+\omega_{2}\xi_{2}=0.
\end{equation}
In terms of Proposition \ref{p2.2},
\begin{equation}\label{4.4a2}
\xi_{1}=P_{\omega_{1}},\;\;\;\xi_{2}=P_{\omega_{2}}\;\;\;\;\text{up to translations}.
\end{equation}
In addition,
\begin{equation}\label{4.4c}
M(\xi_{1})=M(P_{\omega_{1}})=M(\xi_{2})=M(P_{\omega_{2}})=m>\int q^{2}dx.
\end{equation}
According to Theorem \ref{t4.3}, it yields that
\begin{equation}\label{4.4d}
\omega_{1}=\omega_{2}.
\end{equation}
It follows that
\begin{equation}\label{4.4e}
P_{\omega_{1}}=P_{\omega_{2}}\;\;\;\;\text{up to translations}.
\end{equation}
Therefore
\begin{equation}\label{4.4e}
\xi_{1}=\xi_{2}\;\;\;\;\text{up to translations}.
\end{equation}
\end{proof}


\begin{theorem}\label{t4.5}
Let $P_{\omega}$ be a positive solution of (\ref{1.2}) and $\omega\in (0,\frac{3}{16})$. Then $M(P_{\omega})$ is strictly increasing on $\omega\in (0,\frac{3}{16})$.
\end{theorem}
\begin{proof}
According to Proposition \ref{p2.3}, it follows that
\begin{equation}\label{4.5a}
M(P_{\omega})\longrightarrow \int q^{2}(x)dx\;\;\;\;\text{as}\;\;\;\;\omega\longrightarrow 0,
\end{equation}
and
\begin{equation}\label{4.5b}
M(P_{\omega})\longrightarrow \infty\;\;\;\;\text{as}\;\;\;\;\omega\longrightarrow \frac{3}{16}.
\end{equation}
\par On the one hand, for given $\omega\in(0,\frac{3}{16})$,  (\ref{1.2}) admits a unique positive solution $P_{\omega}$. Thus it follows that $m=\int P_{\omega}^{2}dx$.
\par On the other hand, for given $m$, by theorem \ref{t4.3} then there exists a unique $\omega\in (0,\frac{3}{16})$ such that $m=M(P_{\omega})$. By (\ref{4.5a}) and (\ref{4.5b}), Theorem \ref{t4.4} establishes a one-to-one mapping from $(\int q^{2}dx, \infty)$ to $(0,\frac{3}{16})$. Therefore one gets that
\begin{equation}\label{4.5c}
M(P_{\omega})\;\;\;\;\text{is strictly increasing on}\;\;\;(0,\frac{3}{16}).
\end{equation}
\end{proof}


\section{Sharp stability of solitons}\label{sec5}

\begin{proposition}\label{p5.1}
(\cite{C2003}) For arbitrary $\varphi_0\in H^1(\mathbb{R}^{2})$, (\ref{1.1}) possesses a unique global solution $\varphi\in \mathcal{C}(\mathbb{R}; H^1(\mathbb{R}^{2}))$ such that $\varphi(0,x)=\varphi_0$. In addition, the solution holds the conservation of mass, energy and momentum, where mass is given by (\ref{1.6}), energy is given by (\ref{1.7}) and momentum is given by $P(u):=\int2 \mathrm{Im}(\bar{u}\nabla u)dx$ for $u\in H^1(\mathbb{R}^{2})$.
\end{proposition}

\begin{theorem}\label{t5.1}
Let $q(x)$ be the unique positive solution of (\ref{1.10}). For $m>\int q^{2}dx$,
define the constrained variational problem
\begin{equation*}
E_{min}(m)=\mathop{\mathrm{inf}}_{\{u\in H^1(\mathbb{R}^{2}),\;\;  M(u)=m\}} E(u).
\end{equation*}
Then $E_{min}(m)$ is solvable. In addition, for arbitrary minimizing sequence $\{u_n\}^{\infty}_{n=1}$ of $E_{min}(m)$, there exists a subsequence still denoted by  $\{u_n\}^{\infty}_{n=1}$  such that for some $\theta\in\mathbb{R}$ and $y\in\mathbb{R}^{2}$
$$
u_n\rightarrow \psi(.+y)e^{i\theta} \ in \ H^1(\mathbb{R}^{2}), \ as\  n\rightarrow\infty.
$$
\end{theorem}
\begin{proof}
At first we show the set $\{u\in H^{1}(\mathbb{R}^{2}),\; M(u)=m\}$ is nonempty. In fact, for any $u\in H^{1}(\mathbb{R}^{2})\backslash\{0\}$, let $M(u)=m_{0}$ and put
\begin{equation}\label{5.2a}
\lambda=\sqrt{\frac{m}{m_{0}}},\;\;\;v=\lambda u.
\end{equation}
Then we have $M(v)=m$. Thus $v\in \{u\in H^{1}(\mathbb{R}^{2}),\; M(u)=m\}$.
\par Next we prove $E_{min}(m)>-\infty$. Indeed, applying the interpolation inequality and Young inequality, we have that for any $\varepsilon>0$,
\begin{equation}\label{5.2b}
\int \lvert u\lvert^{4}dx\leq \varepsilon \int \lvert u\lvert^{6}dx+\varepsilon^{-1}\int \lvert u\lvert^{2}dx.
\end{equation}
Thus from (\ref{1.7}), it follows that
\begin{equation}\label{5.2c}
E(u)\geq \frac{1}{2}\int \lvert \nabla u\lvert^{2}dx+(\frac{1}{6}-\frac{\varepsilon}{4})\int \lvert u\lvert^{6}dx-\frac{1}{4\varepsilon}\int \lvert u\lvert^{2}dx.
\end{equation}
Take $0<\varepsilon<\frac{2}{3}$, then we have
\begin{equation}\label{5.2d}
E(u)\geq -\frac{1}{4\varepsilon}\int \lvert u\lvert^{2}dx=-\frac{1}{4\varepsilon}m>-\infty.
\end{equation}
This implies that $E_{min}(m)>-\infty$.\\
\par Now we prove $E_{min}(m)<0$. In fact, for $\lambda>0$, $u\in H^{1}(\mathbb{R}^{2})$, let
\begin{equation}\label{5.2e}
u_{\lambda}(x)=\lambda u(\lambda x),
\end{equation}
then we have
\begin{equation}\label{5.2f}
M(u_{\lambda}(x))=M(u)=m,
\end{equation}

\begin{equation}\label{5.2g}
E(u_{\lambda}(x))=\lambda^{2}\Big[\frac{1}{2}\int \lvert \nabla u\lvert^{2}dx-\frac{1}{4}\int \lvert u\lvert^{4}dx+\frac{1}{6}\lambda^{2}\int \lvert u\lvert^{6}dx \Big].
\end{equation}
Since $\|u\|_{L^{2}}^{2}>\|q\|_{L^{2}}^{2}$, then we have
\begin{equation}\label{5.2h}
\int \lvert \nabla u\lvert^{2}dx-\frac{1}{2}\int \lvert u\lvert^{4}dx<0.
\end{equation}
Therefore, one can select $\lambda>0$ sufficiently small such that
\begin{equation}\label{5.2i}
E(u_{\lambda})<0\;\;\;\;\text{for}\;\;\; \lambda\rightarrow 0+.
\end{equation}
This implies that
\begin{equation}\label{5.2ia}
E_{min}(m)<0.
\end{equation}
\par Let $\{u_{n}\}$ be a minimizing sequence of $E_{min}(m)$, then we have
\begin{equation}\label{5.2j}
M(u_{n})=m, \;\;\;\; E(u_{n})\rightarrow E_{min}(m),\;\;\;\text{as}\;\;\;n\rightarrow\infty.
\end{equation}
By the definition of limits, there exists $c>0$ such that
\begin{equation}\label{5.2k}
E(u_{n})\leq E_{min}(m)+c,\;\;\;\text{as}\;\;\;\; n\geq 1.
\end{equation}
By (\ref{5.2c}), we see that
\begin{equation}\label{5.2l}
\frac{1}{2}\int \lvert \nabla u_{n}\lvert^{2}dx\leq E(u_{n})+\frac{1}{4\varepsilon}\int \lvert u_{n}\lvert^{2}dx+1\leq E_{min}(m)+\frac{1}{4\varepsilon}m+1,
\end{equation}
which shows that $\{u_{n}\}$ is bounded in $H^{1}(\mathbb{R}^{2})$. Now we apply the profile decomposition theory (see \cite{HK2005}) to the minimizing sequence $\{u_n\}_{n=1}^{\infty}$. Then there exists a subsequence still denoted by $\{u_n\}_{n=1}^{\infty}$ such that
\begin{equation}\label{5.2m}
u_n(x)=\sum_{j=1}^{l}U_n^j(x)+u_n^l,
\end{equation}
where $U^j_n(x):=U^j(x-x^j_n)$ and $u^l_n:=u^l_n(x)$ satisfies
\begin{equation}\label{5.2n}
\lim_{l\rightarrow\infty}\limsup_{n\rightarrow\infty}\lvert\lvert u^l_n\rvert \rvert_{L^q(\mathbb{R}^2)}=0\quad with \quad  q\in [2,+\infty).
\end{equation}
Moreover, we have the following estimations as $n\rightarrow\infty$:
\begin{equation}\label{5.2o}
\lvert\lvert u_n\rvert \rvert^2_{L^2}=\sum^l_{j=1}\lvert\lvert U_n^j\rvert \rvert^2_{L^2}+\lvert\lvert u_n^l\rvert \rvert^2_{L^2}+o(1),
\end{equation}
\begin{equation}\label{5.2p}
\lvert\lvert \nabla u_n\rvert \rvert^2_{L^2}=\sum^l_{j=1}\lvert\lvert\nabla U_n^j\rvert \rvert^2_{L^2}+\lvert\lvert\nabla u_n^l\rvert \rvert^2_{L^2}+o(1),
\end{equation}
\begin{equation}\label{5.2q}
\lvert\lvert u_n\rvert \rvert^4_{L^4(\mathbb{R}^2)}=\sum^l_{j=1}\lvert\lvert U_n^j\rvert \rvert^4_{L^4(\mathbb{R}^2)}+\lvert\lvert u_n^l\rvert \rvert^4_{L^4(\mathbb{R}^2)}+o(1),
\end{equation}
\begin{equation}\label{5.2r}
\lvert\lvert u_n\rvert \rvert^6_{L^6(\mathbb{R}^2)}=\sum^l_{j=1}\lvert\lvert U_n^j\rvert \rvert^6_{L^6(\mathbb{R}^2)}+\lvert\lvert u_n^l\rvert \rvert^6_{L^6(\mathbb{R}^2)}+o(1).
\end{equation}
Thus we have that
\begin{equation}\label{5.2s}
E(u_n)=\sum^l_{j=1}E(U^j_n)+E(u^l_n)+o(1) \quad as \quad n\rightarrow \infty.
\end{equation}
For $j=1,...,l$, let
\begin{equation}\label{5.2t}
\lambda_j=\frac{\lvert\lvert u_n\rvert \rvert_{L^2}}{\lvert\lvert U^j_n\rvert \rvert_{L^2}},\quad
\lambda^l_n=\frac{\lvert\lvert u_n\rvert \rvert_{L^2}}{\lvert\lvert u^l_n\rvert \rvert_{L^2}}.
\end{equation}
By (\ref{5.2o}) we have that $\lambda_j\geq 1$ and $\lambda^l_n\geq1$. In addition, from the convergence of $\mathop{\sum}^l\limits_{j=1}\lvert\lvert U^j_n\rvert \rvert^2_{L^2}$, there exists a $j_0\geq 1 $ such that
\begin{equation}\label{5.2u}
\mathop{\inf}\limits_{j\geq1}\lambda_j=\lambda_{j_0}
=\frac{\lvert\lvert u_n\rvert \rvert_{L^2}}{\lvert\lvert U^{j_0}_n\rvert \rvert_{L^2}}.
\end{equation}
For $j=1,...,l$, put
\begin{equation}\label{5.2v}
\widetilde{U}^j_n=U^j_n(\lambda^{-1}_j\ x), \quad \widetilde{u}^l_n=u^l_n((\lambda^{l}_n)^{-1}\ x).
\end{equation}
Then we have that
\begin{equation}\label{5.2w}
\lvert\lvert \widetilde{U}^j_n\rvert \rvert^2_{L^2}=m=\lvert\lvert \widetilde{u}^l_n\rvert \rvert^2_{L^2},
\end{equation}

\begin{equation}\label{5.2x}
E(U^j_n)=\frac{E(\widetilde{U}^j_n)}{\lambda^2_j}+\frac{1-\lambda^{-2}_j}{2}\int\lvert \nabla U^j_n\rvert^2dx,
\end{equation}

\begin{equation}\label{5.2y}
E(u^l_n)=\frac{E(\widetilde{u}^l_n)}{(\lambda^l_n)^2}+\frac{1-(\lambda^l_n)^{-2}}{2}\int\lvert \nabla u^l_n\rvert^2dx.
\end{equation}
Thus we deduce that as $n\rightarrow\infty$ and $l\rightarrow\infty$,
\begin{equation}\label{5.2z}
E(u_n)\geq E_{min}(m)+\mathop{\inf}\limits_{j\geq1}(\frac{1-\lambda^{-2}_j}{2})\sum^l_{j=1}\int\lvert\nabla U^j_n\rvert^2dx+\frac{1-(\lambda^l_n)^{-2}}{2}\int\lvert\nabla u^l_n\rvert^2dx+o(1).
\end{equation}
Let $\beta=min\{\lambda_{j_0},\lambda^l_n\}$. Then from (\ref{5.2z}), we have that
\begin{equation}\label{5.2aa}
E(u_n)\geq E_{min}(m)+\frac{1-\beta^{-2}}{2}\|\nabla u_{n}\|_{L^{2}}^{2}+o(1).
\end{equation}
Since $\{u_{n}\}$ is bounded in $H^{1}(\mathbb{R}^{2})$, then it follows that
\begin{equation}\label{5.2ab}
E(u_{n}) \geq E_{min}(m)+\frac{1-\beta^{-2}}{2}c,
\end{equation}
where $c$ is a positive constant. For (\ref{5.2ab}), we get
\begin{equation}\label{5.2ac}
E_{min}(m) \geq E_{min}(m)+\frac{1-\beta^{-2}}{2}c.
\end{equation}
It follows that $\beta\leq1$. Thus we get that
\begin{equation}\label{5.2ad}
\lvert\lvert u_n\rvert \rvert_{L^2}\leq \lvert\lvert U^{j_0}_n\rvert \rvert_{L^2} \quad or \quad
\lvert\lvert u_n\rvert \rvert_{L^2}\leq \lvert\lvert u^{l}_n\rvert \rvert_{L^2}.
\end{equation}
If $\lvert\lvert u_n\rvert \rvert_{L^2}\leq\lvert\lvert u^{l}_n\rvert \rvert_{L^2}$, one deduces that
\begin{equation}\label{5.2ae}
\lvert\lvert u_n\rvert \rvert^4_{L^4(\mathbb{R}^2)}\rightarrow0 \quad as \quad n\rightarrow\infty,
\end{equation}
\begin{equation}\label{5.2af}
\lvert\lvert u_n\rvert \rvert^6_{L^6(\mathbb{R}^2)}\rightarrow0 \quad as \quad n\rightarrow\infty.
\end{equation}
Thus
\begin{equation}\label{5.2ag}
E_{min}(m)=\lim_{n\rightarrow\infty}E(u_{n})=\lim_{n\rightarrow\infty}\frac{1}{2}\|\nabla u_{n}\|_{L^{2}}^{2}\geq 0,
\end{equation}
which contradicts with $E_{min}(m)<0$. Therefore it is necessary that
\begin{equation}\label{5.2ah}
\lvert\lvert u_n\rvert \rvert_{L^2}\leq\lvert\lvert U^{j_0}_n\rvert \rvert_{L^2}.
\end{equation}
Thus we get that
\begin{equation}\label{5.2ai}
\lvert\lvert u_n\rvert \rvert^2_{L^2}=\lvert\lvert U^{j_0}_n\rvert \rvert^2_{L^2}, \quad
\lvert\lvert \nabla u_n\rvert \rvert^2_{L^2}=\lvert\lvert \nabla U^{j_0}_n\rvert \rvert^2_{L^2}.
\end{equation}
By (\ref{5.2m}), (\ref{5.2o}) and (\ref{5.2p}), it follows that
\begin{equation}\label{5.2aj}
u_n(x)=U^{j_0}_n(x)=U^{j_0}(x-x^{j_0}_n).
\end{equation}
Let
\begin{equation}\label{5.2ak}
u_n\rightharpoonup v \ in  \ H^1(\mathbb{R}^2).
\end{equation}
Then
\begin{equation}\label{5.2al}
u_n\rightarrow v \ a.e.\ in \ \mathbb{R}^2.
\end{equation}
Thus there exists some fixed $x^{j_0}_n$, denoted by $x^{j_0}$ such that
\begin{equation}\label{5.2am}
v=U^{j_0}(x-x^{j_0}):=U^{j_0}, \ a.e. \ in\ \mathbb{R}^2.
\end{equation}
Therefore we have that
\begin{equation}\label{5.2an}
u_n\rightarrow U^{j_0}\ \  in \ \ H^1(\mathbb{R}^2).
\end{equation}
It is clear that $I(U^{j_0})=0$. Thus $U^{j_0}$ is a minimizer of (\ref{4.0a}). Then for some 
\begin{equation}\label{5.2ao}
U^{j_0}=\psi(\cdot+y)e^{i\theta}.
\end{equation}
Thus one deduces that
\begin{equation}\label{5.2ap}
u_n\rightarrow \psi(\cdot+y)e^{i\theta}\ \  in \ \ H^1(\mathbb{R}^2).
\end{equation}
\end{proof}

Now we complete the proof of Theorem \ref{t1.1}.
\begin{proof}
In terms of Theorem \ref{t5.1} and Theorem \ref{t4.4}, one gets that Theorem \ref{t1.1} is true.
\end{proof}

\begin{theorem}\label{t5.2}
Let $\psi$ be the unique positive minimizer of (\ref{4.0a}) up to translations with $m>\int q^{2}dx$. Then the set of all minimizers of $E_{min}(m)$ is that
\begin{equation*}
S_m=\{e^{i\theta}\psi(\cdot+y), \; \theta\in\mathbb{R},\; y\in\mathbb{R}^2\}.
\end{equation*}
\end{theorem}
\begin{proof}
Suppose that $u$ is a minimizer of the variational problem (\ref{1.11}). One has that
\begin{equation}\label{5.3a}
u=\rvert u\rvert e^{i\theta}\;\;\;\text{for some}\;\;\;\theta\in \mathbb{R}.
\end{equation}
Since for $u\in H^{1}(\mathbb{R}^{2})$,
\begin{equation}\label{5.3b}
\int\rvert\nabla u\rvert^{2}dx\geq \int\rvert\nabla \rvert u\rvert\rvert^{2}dx,
\end{equation}
it follows that
\begin{equation}\label{5.3c}
E(u)\geq E(\rvert u\rvert).
\end{equation}
It yields that $\rvert u\rvert$ is also a
minimizer of the variational problem (\ref{1.11}). By Theorem \ref{t4.4}, one implies that
\begin{equation}\label{5.3d}
\rvert u\rvert=\psi \;\;\;\;\text{up to translations}.
\end{equation}
It follows that
\begin{equation}\label{5.3e}
u\in S_{m}=\{ e^{i\theta}\psi(\cdot+y),\ \ \theta\in \mathbb{R},\;\;\;y\in\mathbb{R}^2\}.
\end{equation}
\end{proof}

Now we complete the proof of Theorem \ref{t1.2}.
\begin{proof}
By Theorem \ref{t1.1}, for $m>\int q^{2}dx$, the variational problem $E_{min}(m)$ possesses a unique positive minimizer $\psi$ up to a translation. Then $\psi$ satisfies the Euler-Lagrange equation (\ref{1.2}) with the Lagrange multipliers $\omega\in(0,\frac{3}{16})$. By Proposition \ref{p5.1} (refer to \cite{C2003}), for arbitrary $\varphi_{0}\in H^{1}$, (\ref{1.1}) with $\varphi(0,x)=\varphi_{0}(x)$ possesses a unique global solution $\varphi(t,x)\in \mathcal{C}(\mathbb{R}, H^{1}(\mathbb{R}^{2}))$. In addition, $\varphi(t,x)$ satisfies the mass conservation $M(\varphi(t,\cdot))=M(\varphi_{0}(\cdot))$ and the energy conservation $E(\varphi(t,\cdot))=E(\varphi_{0}(\cdot))$ for all $t\in\mathbb{R}$. Now arguing by contradiction.

If the conclusion of Theorem \ref{t1.2} does not hold, then there exist $\varepsilon>0$, a sequence $(\varphi_{0}^{n})_{n\in N^{+}}$ such that
\begin{equation}\label{5.4a}
\mathop{\inf}\limits_{\{\theta\in\mathbb{R},y\in\mathbb{R}^2\}}\lvert\lvert \varphi_0^n-e^{i\theta}\psi(\cdot+y)\rvert \rvert_{H^1}<\frac{1}{n},
\end{equation}
and a sequence $(t_n)_{n\in \mathbb{N}^+}$ such that
\begin{equation}\label{5.4b}
\mathop{\inf}\limits_{\{\theta\in\mathbb{R},y\in\mathbb{R}^2\}}\lvert\lvert \varphi_n(t_n,\cdot)-e^{i\theta}\psi(\cdot+y)\rvert \rvert_{H^1}\geq\varepsilon,
\end{equation}
where $\varphi_n$ denotes the global solution of (\ref{1.1}) with $\varphi(0,x)=\varphi_0^n$. From (\ref{5.4a}) it yields that for some $\theta\in\mathbb{R}$, $y\in\mathbb{R}^2$,
\begin{equation}\label{5.4c}
\varphi^n_0\rightarrow e^{i\theta}\psi(\cdot+y), \quad in \quad H^1(\mathbb{R}^2), \quad n\rightarrow\infty.
\end{equation}
Thus we have that
\begin{equation}\label{5.4d}
\int\lvert \varphi^n_0\rvert^2dx \rightarrow \int  \psi^2dx,\;\; \quad E(\varphi^n_0)\rightarrow E(\psi),\;\; n\rightarrow\infty.
\end{equation}
Since for $n\in N^{+}$ ,
\begin{equation}\label{5.4e}
\int\lvert \varphi_{n}(t_{n},\cdot)\lvert^{2}dx=\int \lvert \varphi_{0}^{n}\lvert^{2}dx,\;\;\; E(\varphi_{n}(t_{n},\cdot))=E(\varphi_{0}^{n}),
\end{equation}
from (\ref{5.4d}) we have that
\begin{equation}\label{5.4f}
\int\lvert \varphi_{n}(t_{n},\cdot)\lvert^{2}dx\rightarrow \int \psi^{2}dx,\;\;\; E(\varphi_{n}(t_{n},\cdot))\rightarrow E(\psi),\;\; n\rightarrow\infty.
\end{equation}
This yields that from (\ref{5.4f})
\begin{equation}\label{5.4g}
\lim_{n\rightarrow\infty}\lvert\lvert \varphi_{n}(t_{n},\cdot)-e^{i\theta}\psi(\cdot+y)\rvert \rvert_{H^1}=0.
\end{equation}
This is contradictory with (\ref{5.4b}).

Therefore Theorem \ref{t1.2} is true.
\end{proof}

\section{Classification of normalized solutions}\label{sec8}

\begin{theorem}\label{t6.1}
Let $u(x)$ be a solution of (\ref{1.2}) with $\omega\in (0,\frac{3}{16})$ and $q(x)$ be the unique positive solution of (\ref{1.10}). Then one yields that
\begin{equation*}
M(u)> \int q^{2}dx\;\;\;\text{for\;\;all}\;\;\; \omega\in (0,\frac{3}{16}).
\end{equation*}
\end{theorem}
\begin{proof}
Since $u(x)$ is the unique positive solution of $(\ref{1.2})$, then one has that
\begin{equation}\label{2.1}
\int\lvert \nabla u \rvert^2 + \frac{2}{3}\lvert u\lvert ^6-\frac{1}{2}\lvert u\lvert ^4dx=0.
\end{equation}
By the sharp Ggaliardo-Nirenberg inequality
\begin{equation}\label{2.2}
\int \lvert u\rvert^4dx\leq 2\frac{\lvert\lvert u\lvert\lvert_{L^{2}(\mathbb{R}^{2})}^{2}}{\lvert\lvert q\lvert\lvert_{L^{2}(\mathbb{R}^{2})}^{2}}\int \lvert \nabla u\rvert^2dx,\;\;\;\;\forall u\in H^{1}(\mathbb{R}^{2}),
\end{equation}
and (\ref{2.1}), one has that
\begin{equation}\label{2.3}
\Big(1-\frac{\lvert\lvert u\lvert\lvert_{L^{2}(\mathbb{R}^{2})}^{2}}{\lvert\lvert q\lvert\lvert_{L^{2}(\mathbb{R}^{2})}^{2}}\Big)\int \lvert \nabla u \rvert^2+\frac{2}{3}\int u^6dx\leq 0.
\end{equation}
This implies that
\begin{equation}\label{2.4}
\lvert\lvert u\lvert\lvert_{L^{2}(\mathbb{R}^{2})}^{2}>\lvert\lvert q\lvert\lvert_{L^{2}(\mathbb{R}^{2})}^{2}.
\end{equation}
\end{proof}

\begin{theorem}\label{t6.2}
Let $q(x)$ be the unique positive solution of (\ref{1.10}). Then when $0<m\leq\int q^{2}dx$, (\ref{1.2}) has no any normalized solutions with the prescribed mass $\int \lvert u\lvert^{2}dx=m$. When $m>\int q^{2}dx$, (\ref{1.2}) has a unique positive normalized solution with the prescribed mass $\int \lvert u\lvert^{2}dx=m$.
\end{theorem}
\begin{proof}
When $0<m\leq\int q^{2}dx$, from Theorem (\ref{t6.1}) it yields that (\ref{1.2}) has no any normalized solutions with the prescribed mass $\int \lvert u\lvert^{2}dx=m$. On the other hand, in terms of Proposition \ref{p2.2}, for $\omega\in (0,\frac{3}{16})$, (\ref{1.2}) possesses a unique positive solution $P_{\omega}(x)$.  For this $\omega$, according to Theorem \ref{t6.1} there exists a $m$ such that
\begin{equation}\label{6.1a}
M(P_{\omega})=m>\int q^{2}dx.
\end{equation}
By Theorem \ref{t5.1}, the variational problem $E_{min}(m)$ possesses a positive minimizer $\psi(x)$. By Theorem \ref{t1.1}, this positive minimizer $\psi(x)$ is unique up to translations. Then there exists a unique Lagrange multiplier $\omega'$ corresponding to $\psi$  such that $\psi$ satisfies (\ref{1.2}). Note that
\begin{equation}\label{6.1b}
M(\psi)=m=M(P_{\omega}).
\end{equation}
By Theorem \ref{t4.3}, one has that $\omega=\omega'$. From Proposition \ref{p2.2},
\begin{equation*}
P_{\omega}(x)=\psi(x)\;\;\;\;\text{up to a translation}.
\end{equation*}
Therefore in terms of Theorem \ref{t4.4}, one deduce that when $m>\int q^{2}dx$, (\ref{1.2}) has a unique positive normalized solutions with the prescribed mass $\int \rvert u\rvert^{2}dx=m$.
\end{proof}

\textbf{Acknowledgment.}

This research is supported by the National Natural Science Foundation of China 12271080, 12571318 and Sichuan Technology Program 25LHJJ0156.

\textbf{Data Availability Statement}

Data sharing not applicable to this article as no datasets were generated or analysed during the current study.

\end{document}